\documentclass{amsart}

\makeatletter
 \def\LaTeX{\leavevmode L\raise.42ex
   \hbox{\kern-.3em\size{\sf@size}{0pt}\selectfont A}\kern-.15em\TeX}
\makeatother

\newcommand{\BibTeX}{{\rm B\kern-.05em{\sc
i\kern-.025emb}\kern-.08em\TeX}}
\newtheorem{col}{Corollary}[section]
\newtheorem{thm}{Theorem}[section]
\newtheorem{lem}[thm]{Lemma}

\theoremstyle{definition}
\newtheorem{defn}{Definition}
\newtheorem{exmp}{Example}
\newtheorem{rem}{Remark}
\numberwithin{equation}{section}

\begin{document}

\title{Sampling in Paley-Wiener spaces on combinatorial graphs}

\author{Isaac Pesenson}
\address{Department of Mathematics, Temple University,
Philadelphia, PA 19122} \email{pesenson@temple.edu}

\keywords{Combinatorial graph, combinatorial Laplace operator,
discrete Paley-Wiener spaces, Shannon sampling, discrete
Plancherel-Polya and Poincare inequalities.} \subjclass{42C99,
05C99, 94A20; Secondary  94A12 }

 \begin{abstract} A notion of Paley-Wiener spaces is introduced
  on  combinatorial graphs. It is shown that functions
  from some of these  spaces are uniquely determined by their values on
  some sets of vertices which are called the uniqueness sets. Such uniqueness
   sets are described in terms of Poincare-Wirtinger-type
   inequalities. A reconstruction algorithm of Paley-Wiener
    functions from uniqueness sets which   uses the idea of frames in Hilbert spaces
  is developed. Special consideration is given to $n$-dimensional lattice,
     homogeneous trees, and eigenvalue and eigenfunction problems on finite graphs.

\end{abstract}

\maketitle

 \section{Introduction and Main Results}

 The goal of the paper is to develop a sampling theory of Paley-Wiener functions
(bandlimited functions) on combinatorial graphs. It is shown that
functions which involve only "low" frequencies can be perfectly
reconstructed from their values on some  subsets of vertices. Note
that on a continuous manifold for any frequency one can construct
a sampling set of sufficient  density which will allow
reconstruction of that frequency.

Let us remind some basic facts from the classical sampling theory.
A function $f\in L_{2}(\mathbb{R})$ is called $\omega$-bandlimited
if its $L_{2}$-Fourier transform
$$
\hat{f}(t)=\int_{-\infty}^{+\infty} f(x)e^{-2\pi i x t}dx
$$
has support in $[-\omega,\omega]$. The Paley-Wiener theorem states
that $f\in L_{2}(\mathbb{R})$ is $\omega$-bandlimited if and only
if
 $f$ is an entire function of exponential type not exceeding $2\pi
 \omega$. $\omega$-bandlimited functions form the
 Paley-Wiener class $PW_{\omega}(\mathbb{R})$
  and often called  Paley-Wiener functions.
The classical sampling theorem says, that if $f$ is
$\omega$-bandlimited then $f$ is completely determined by its
values at points $j/2\omega, j\in \mathbb{Z}$, and can be
reconstructed in a stable way from the samples $f(j/2\omega)$,
i.e.
\begin{equation}
 f(x)= \sum_{j\in \mathbb{Z}} f\left(\frac{j}{2\omega}\right)\frac{\sin(2\pi\omega
(x-j/2\omega))}{2\pi \omega (x-j/2\omega )},
\end{equation}
where convergence is understood in the  $L_{2}$-sense. The formula
(1.1)  involves regularly spaced points $j/2\omega, j\in
\mathbb{Z}$. If one would like to consider irregular sampling at a
sequence of points $\{x_{j}\}$ and still have a stable
reconstruction from the samples $f(x_{j})$ then the following
Plancherel-Polya inequality \cite{PP1}, \cite{PP2}, should hold
true
\begin{equation}
C_{1}\left(\sum_{j\in
\mathbb{Z}}|f(x_{j})|^{2}\right)^{1/2}\leq\|f\|_{L_{2}(\mathbb{R})}\leq
C_{2}\left(\sum_{j\in \mathbb{Z}}|f(x_{j})|^{2}\right)^{1/2}.
\end{equation}
 There
is a classical result of Duffin and Schaeffer \cite{DS}, that the
inequalities (1.2) imply existence of a \textit{dual frame}
$\left\{\Theta_{j}\right\}$  which consists of functions in
$PW_{\omega}(\mathbb{R})$ such that any function $f\in
PW_{\omega}(\mathbb{R})$ can be reconstructed according to the
following formula
\begin{equation}
f(x)=\sum_{j\in \mathbb{Z}}f(x_{j})\Theta_{j}(x).
\end{equation}
A similar   approach can be developed for the  Paley-Wiener spaces
$PW_{\omega}(\mathbb{R}^{d})$.  The formula (1.3) is a
generalization of the formula (1.1) because it can be used for
non-uniformly spaced sets of sampling points.

The theory of irregular sampling was very active for many years
 \cite{B64}, \cite{BM67}, \cite{DS}, \cite{Lan67},
\cite{PW34}, \cite{OS}, \cite{LS02}, \cite{Z}. Some of the ideas
and methods of this theory were recently extended to the cases of
Riemannian manifolds, symmetric spaces, groups, and quantum graphs
\cite{EEKK1}, \cite{EEKK2},
 \cite{FP1}, \cite{FP2}, \cite{F}, \cite{FG}, \cite{Pes98}- \cite{Pes062}.

In the present article the
  Paley-Wiener spaces are introduced on  combinatorial
  graphs  and
  a corresponding sampling theory is developed which  resembles the classical
  one. Namely it is shown that Paley-Wiener functions of  low type
  are uniquely determined by their values on certain subgraphs (uniqueness sets) and
  can be reconstructed from such sets in a stable way. A
  description and examples of some of uniqueness sets are given.
 A reconstruction method is presented which
  gives a formula of the type (1.3) in terms of dual frames.
  More detailed consideration is given
  to particularly interesting cases
  of $n$-dimensional lattice $\mathbb{Z}^{n}$, homogeneous trees and finite graphs.
   Applications to eigenvalue and eigenfunction problems
on finite graphs are also considered.

It seems that our results can find different applications in
signal analysis, imaging, learning theory and discrete tomography
\cite{HK}, \cite{SS1}, \cite{SS2}.

 We know just three papers \cite{FrT}, \cite{G},
\cite{MSW} in which authors consider sampling on $\mathbb{Z}^{n}$
and on $\mathbb{Z}_{N}$, but our approach to the problem and our
results are very different from the methods of these papers.

 The following is a summary of main
notions and results. We consider finite or infinite and in this
case countable connected graphs $G=(V(G),E(G))$,  where $V(G)$ is
its set of vertices and $E(G)$ is its set of edges. We consider
only simple (no loops, no multiple edges) undirected unweighted
graphs. A number of vertices adjacent to a vertex $v$ is called
the degree of $v$ and denoted by $d(v)$. We assume that degrees of
all vertices are bounded from above and we use the notation
$$
d(G)=\sup_{v\in V(G)}d(v).
$$
 The space $L_{2}(G),$ is the Hilbert space of all
complex-valued functions $f:V(G)\rightarrow \mathbb{C}$ with the
following inner product
$$
\left<f,g\right>= \sum_{v\in V(G)}f(v)\overline{g(v)}
$$
and the following norm
\begin{equation}
\|f\|=\|f\|_{0}=\left(\sum_{v\in V(G)}|f(v)|^{2}\right)^{1/2}.
\end{equation}

The discrete Laplace operator $\mathcal{L}$ is defined by the
formula  \cite{Ch}
\begin{equation}
\mathcal{L}f(v)=\frac{1}{\sqrt{d(v)}}\sum_{v\sim
u}\left(\frac{f(v)}{\sqrt{d(v)}}-\frac{f(u)}{\sqrt{d(u)}}\right),
f\in L_{2}(G),
\end{equation}
where $v\sim u$ means that $v, u\in V(G)$ are connected by an
edge. It is known that the Laplace operator $\mathcal{L}$ is a
bounded operator in $L_{2}(G)$ which is self-adjoint and positive
definite.  Let $\sigma(\mathcal{L})$ be the spectrum of a
self-adjoint positive definite operator $\mathcal{L}$ in
$L_{2}(G)$. In what follows we will use the notations
$$
\omega_{\min}=\inf_{\omega \in \sigma(\mathcal{L})}\omega,
\omega_{\max}=\sup_{\omega \in \sigma(\mathcal{L})}\omega.
$$

According to the spectral theory \cite{BS} there exist a
direct integral of Hilbert spaces $X=\int X(\tau)dm (\tau )$ and a
unitary operator $F$ from $L_{2}(G)$ onto $X$, which transforms
the domain of $\mathcal{L}^{s}, s\geq 0,$ onto $X_{s}=\{x\in
X|\tau^{s}x\in X \}$ with norm
$$
\|x(\tau )\|_{X_{s}}= \left (\int_{\sigma(\mathcal{L})} \tau^{2s}
\|x(\tau )\|^{2}_{X(\tau)} d m(\tau) \right )^{1/2}
$$
and $F(\mathcal{L}^{s} f)=\tau ^{s} (Ff)$.
 We introduce the following notion of discrete Paley-Wiener spaces.
\begin{defn}
Given an $\omega\geq 0$ we will say that a function $f$ from
$L_{2}(G)$ belongs to the Paley-Wiener space $PW_{\omega}(G)$ if
its "Fourier transform" $Ff$ has support in
 $[0, \omega ] $.
\end{defn}

\begin{rem}
To be more consistent with the definition of the classical
Paley-Wiener spaces we should consider the interval
$[0,\omega^{2}]$ instead of $[0,\omega]$. We prefer our choice
because it makes formulas and notations simpler.
\end{rem}

Since the operator $\mathcal{L}$ is bounded every function from
$L_{2}(G)$ belongs to a certain  Paley-Wiener space
$PW_{\omega}(G)$ for some $\omega\in \sigma(\mathcal{L})$ and we
have the following stratification
\begin{equation}
L_{2}(G)=PW_{\omega_{\max}}(G)=\bigcup_{ \omega\in
\sigma(\mathcal{L})}PW_{\omega}(G), PW_{\omega_{1}}(G)\subseteq
PW_{\omega_{2}}(G), \omega_{1}<\omega_{2}.
\end{equation}

Different properties of the spaces $PW_{\omega}(G)$ and in
particular a generalization of the Paley-Wiener Theorem are
collected in the Theorem 2.1.

In Theorems 5.2, 5.5 and 5.9 it is  shown that if a graph $G$ is
an $n$-dimensional lattice or a homogeneous tree then  a
difference between two Paley-Wiener spaces can be recognized only
at infinity.

\begin{defn}
We say that a set of vertices $U\subset V(G)$ is a uniqueness set
for a space  $PW_{\omega}(G), \omega>0,$ if for any two functions
from $PW_{\omega}(G)$ the fact that they coincide on $U$ implies
that they coincide on $V(G)$.
\end{defn}
 For a subset $S\subset V(G)$ (finite or infinite) the
notation $L_{2}(S)$ will denote the space of all functions from
$L_{2}(G)$ with support in $S$:
$$
L_{2}(S)=\{\varphi\in L_{2}(G), \varphi(v)=0, v\in V(G)\backslash
S\}.
$$
\begin{defn}
We say that a set of vertices $S\subset V(G)$ is a $\Lambda$-set
if for any $\varphi\in L_{2}(S)$ it admits a Poincare  inequality
with a constant $\Lambda=\Lambda(S)>0$
\begin{equation}
\|\varphi\|\leq \Lambda\|\mathcal{L}\varphi\|, \varphi\in
L_{2}(S).
\end{equation}
The infimum of all $\Lambda>0$ for which $S$ is a $\Lambda$-set
will be called the Poincare constant of the set $S$ and denoted by
$\Lambda(S)$.
\end{defn}
 In the
Theorem 3.4 we give several estimates of the constant $\Lambda$
for finite sets. The role of $\Lambda$-sets is explained in the
following Theorem which is one of the main observations made in
this paper.

\begin{thm}
If a set $S\subset V(G)$  is a $\Lambda$-set,  then
 the set  $U=V(G)\backslash S$ is a
uniqueness set for any space $PW_{\omega}(G)$ with $0< \omega<
1/\Lambda$.
\end{thm}

 Since $L_{2}(G)=PW_{\omega_{\max}}(G)$
one cannot expect that non-trivial uniqueness sets there exist for
functions from every Paley-Wiener subspace $PW_{\omega}(G)$ with
any $\omega_{\min}\leq \omega\leq\omega_{\max}$. Indeed, otherwise
it would mean that certain subsets of vertices can be removed from
a graph without changing spectral properties of $\mathcal{L}$. But
it is reasonable to expect that uniqueness sets exist for
Paley-Wiener spaces $PW_{\omega}(G)$ with relatively small
$\omega>0$.

It is  shown in this article that   for every graph $G$ there
exists a constant $1<\Omega_{G}<\omega_{\max}$ such that for
$0<\omega<\Omega_{G}$ functions from $PW_{\omega}(G)$  can be
determined by using their values only on certain subsets of
vertices. For instance,  one can show that when $S$ is a vertex
$v$ in a graph $G$ then the constant $\Lambda=\Lambda(v)$ in (1.7)
equals
$$
\frac{1}{\sqrt{2}}\leq
\Lambda(v)=\frac{1}{\sqrt{1+\frac{1}{d(v)}\sum_{w\sim
v}\frac{1}{d(w)}}}\leq \frac{1}{\sqrt{1+\frac{1}{d(G)}}}.
$$
According to the  Theorem 1.1 it shows that for any graph $G$
spaces $PW_{\omega}(G)$ with
\begin{equation}
0<\omega<\sqrt{1+\frac{1}{d(G)}}=\Omega_{G}>1
\end{equation}
have non-trivial uniqueness sets.

 We have to emphasize
that our results are not trivial only for graphs for which
interval
 $(0,\Omega_{G}]$ has non-empty intersection with the spectrum $\sigma(\mathcal{L})$.
Bellow are some  examples for which this condition is satisfied.

\begin{enumerate}

\item $n$-dimensional lattices $\mathbb{Z}^{n}$ for which the
spectrum $\sigma(\mathcal{L})$ is the entire interval $[0,2]$ and
$ \Omega_{G}=\sqrt{1+2^{-n}}>1.$

\item Infinite countable  graphs with bounded vertex degrees which
have polynomial growth. For such graphs $\omega_{\min}=\inf
\sigma(\mathcal{L})$ is always zero \cite{M}.

\item  Homogeneous trees of order $q+1$ for which
$$
\sigma(\mathcal{L})=\left[1-\eta(q),1+\eta(q)\right],
\Omega_{G}=\sqrt{1+(q+1)^{-1}}>1, \eta(q)=\frac{2\sqrt{q}}{q+1}.
$$

\item Finite and not complete graphs $G$ for all of which  the
first nonzero eigenvalue $\lambda_{G}$ satisfies $\lambda_{G}\leq
1<\Omega_{G}$. For example for any planar graph with $n$ vertices
the following estimate is known \cite{M}
$$
\lambda_{G}\leq \frac{12d_{\max}}{\sqrt{n/2}-6},
$$
which shows that for a fixed $d_{\max}$ and large $n$ the
eigenvalue $\lambda_{G}$ is close to zero.

\end{enumerate}

Given a proper subset of vertices $S\subset V(G)$ its
\textit{vertex boundary} $bS$ is the set of all vertices in $V(G)$
which are not in $S$ but adjacent to a vertex in S
$$
b S=\left\{v\in V(G)\backslash S:\exists \{u,v\}\in E(G), u\in
S\right\}.
$$

If a graph $G=(V(G),E(G))$ is connected and $S$ is a proper subset
of $V$ then the vertex boundary $b S$ is not empty. For a finite
set $S$ consider the set $S\cup bS=\overline{S}$ as an induced
graph. It means that the  graph $\overline{S}$  is determined by
all edges of $G$ with both endpoints in $\overline{S}=S\cup bS$.

\begin{defn}
If $S$ is a finite proper subset of vertices then the notation
$\Gamma(S)$ will be used for a graph constructed in the following
way. Take two copies of the induced graph $\overline{S}=S\cup bS$,
which we will denote as $\overline{S}_{1}$ and $\overline{S}_{2}$
and identify  every vertex $v\in b S\subset \overline{S}_{1}$ with
"the same" vertex $v\in b S\subset \overline{S}_{2}$.
\end{defn}

The following statement summarizes some of our main results
(Theorem 3.7 and Theorem 4.1).
\begin{thm}
For a given $\omega_{\min}<\omega< \Omega_{G}$ consider a set of
vertices  $S=\bigcup S_{j}$ with the following properties:
\begin{enumerate}
\item for every $S_{j}\subset V(G)$ the following inequality takes
place
\begin{equation}
\lambda_{1}(\Gamma(S_{j}))>\omega,
\end{equation}
where $\lambda_{1}(\Gamma(S_{j}))$ is the first positive
eigenvalue of the graph $\Gamma(S_{j})$;

\item the sets $\overline{S}_{j}=S_{j}\cup bS_{j}$ are disjoint.

 Then the following holds true:

 \item
the set $U=V(G)\setminus S$ is a uniqueness set for the space
$PW_{\omega}(G)$;

\item  there exists a frame of functions $\{\Theta_{u}\}_{u\in U}$
in the space $PW_{\omega}(G)$ such that the following
reconstruction formula holds true for all $f\in PW_{\omega}(G)$
\begin{equation}
f(v)=\sum_{u\in U}f(u)\Theta_{u}(v), v\in V(G).
\end{equation}
\end{enumerate}
\end{thm}
Note that the last formula is an analog of the formula (1.3).

Let us illustrate some of our   main results in the case of a line
graph $\mathbb{Z}$. In this case the spectrum of the Laplace
operator is the interval $[0, 2]$ and  the constant
$\Omega_{\mathbb{Z}}$ which is defined in (1.9) equals
$\sqrt{3/2}$.
\begin{thm}
For any finite subset $S$ of successive vertices of the graph
$\mathbb{Z}$ the following inequality holds true
\begin{equation}
\|\varphi\|\leq
\frac{1}{2}\sin^{-2}\frac{\pi}{2|S|+2}\|\mathcal{L}\varphi\|,
\varphi\in L_{2}(S).
\end{equation}
It implies that for a given $0<\omega<\sqrt{3/2}$ every function
$f\in PW_{\omega}(G)$ is uniquely determined by its values on a
set $U=V(G)\backslash S$, where $S$ is a finite or infinite union
of
 disjoint sets $\{S_{j}\}$ of successive vertices such
that
\begin{enumerate}

\item the sets $\overline{S}_{j}=S_{j}\cup b S_{j}$ are disjoint

and

\item
$$
|S_{j}|< \frac{\pi}{2\arcsin\sqrt{\frac{ \omega}{2}}}-1, j\in
\mathbb{N}.
$$
\end{enumerate}

 Moreover, there exists a frame of functions
$\{\Theta_{u}\}_{u\in U}$ in the space $PW_{\omega}(G)$ such that
the following reconstruction formula holds true for all $f\in
PW_{\omega}(G)$
\begin{equation}
f(v)=\sum_{u\in U}f(u)\Theta_{u}(v), v\in V(G).
\end{equation}

\end{thm}

When $\omega$ is really small we obtain the following "Nyquist
rate" of sampling
\begin{equation}
|S_{j}|\sim \frac{\pi}{\sqrt{2\omega}}-1\sim
\frac{\pi}{\sqrt{2\omega}} .
\end{equation}

Using this Theorem it is easy to estimate that if, for example,
$\omega=0.5$  then there are  uniqueness sets for
$PW_{0.5}(\mathbb{Z})$ that contain about 50 percent of points of
any interval of length $4k$ in $\mathbb{Z}$; if $\omega=0.1$, then
there are uniqueness sets for $PW_{0.1}(\mathbb{Z})$ that contain
about 25 percent of points of any interval of length $8k$ in
$\mathbb{Z}$; if $\omega=0.01$,  then there are  uniqueness sets
for $PW_{0.01}(\mathbb{Z})$ that contain about 9 percent of points
of any interval in $\mathbb{Z}$.

\begin{rem}
If one will compare our relations (1.13) between $S_{j}$ and
$\omega$ with the corresponding relations between a rate of
sampling and frequency $\omega$ for classical Paley-Wiener spaces
he can be confused with the presence of a square root for
$\omega$. It is because there is a certain discrepancy between our
and classical definitions of Paley-Wiener spaces (see Remark 1).
\end{rem}

 \begin{rem}
Our inequalities (1.11) are similar to the inequalities which were
obtained in \cite{FTT}. Note that the proofs in \cite{FTT} relay
on the knowledge  of eigenvalues of certain Hermitian matrices
which  were calculated  by a physicist D. E. Rutherford in
\cite{R} in connection with some problems in physics and
chemistry. What is really interesting that D. E. Rutherford
considered  graphs as models of some physical systems.
\end{rem}
 A similar result holds true for a lattice
$\mathbb{Z}^{n}$ of any dimension. The spectrum of the Laplace
operator is $[0, 2]$ and
$\Omega_{\mathbb{Z}^{n}}=\sqrt{(2n+1)/2n}$. We have the following
sampling Theorem.

\begin{thm}For a given $0<\omega<\sqrt{(2n+1)/2n}$ every function $f\in
PW_{\omega}(G)$ is uniquely determined by its values on a set
$U=V(G)\backslash S$, where $S$ is a finite or infinite union of
 disjoint $N_{1,j}\times N_{2,j}\times...\times N_{n,j}$ "rectangular solids"
   $\{S_{j}\}$ of vertices such
that

\begin{enumerate}

\item the sets $\overline{S}_{j}=S_{j}\cup b S_{j}$ are disjoint

and

\item the following inequality holds true for all $j$
$$
\omega< 4\min\left(\sin \frac{\pi}{2N_{1,j}+2},\sin
\frac{\pi}{2N_{2,j}+2},...,\sin
\frac{\pi}{2N_{n,j}+2}\right)=C_{j}.
$$

\end{enumerate}

 Moreover, there exists a frame of functions
$\{\Theta_{u}\}_{u\in U}$ in the space $PW_{\omega}(G)$ such that
the following reconstruction formula holds true for all $f\in
PW_{\omega}(G)$
$$
f(v)=\sum_{u\in U}f(u)\Theta_{u}(v), v\in V(G).
$$

\end{thm}

We also consider Paley-Wiener spaces on homogeneous trees and in
Chapter 6 give some applications to eigenvalue and eigenfunction
approximations on finite graphs. Bellow we  formulate some
consequences of our results about finite graphs.

 Thus we assume that a graph $G$ has $N$ vertices and eigenvalues of
 the Laplace operator $\mathcal{L}$ are $0=\lambda_{0}<\lambda_{1}\leq \lambda_{2}
 \leq...\leq \lambda_{N-1}$. Let
$\mathcal{N}[0,\omega)$ denote the number of eigenvalues of
$\mathcal{L}$ in $[0,\omega)$ and  $\mathcal{N}[\omega,
\omega_{\max}]$ is a number of eigenvalues of $\mathcal{L}$ in
$[\omega, \omega_{\max}]$. The notation $\mathcal{P}(\Lambda)$ is
used for all sets of vertices $S\subset V(G)$ which satisfy (1.7).
The next Corollary gives a certain information about distribution
of eigenvalues of $\mathcal{L}$.

\begin{col}
For any set $S$ which satisfies (1.7) the following inequalities
hold true
$$
\mathcal{N}[0,1/\Lambda)\leq|V(G)|-|S|,
$$
and
$$
\mathcal{N}[1/\Lambda, \omega_{\max}]\geq|S|.
$$
In particular if $M=\max_{S\in \mathcal{P}(\Lambda)}|S|$, then
$$
\mathcal{N}[0,1/\Lambda)\leq|V(G)|-M.
$$
\end{col}
To illustrate this result let us consider  a cycle graph
$C_{100}=\{1, 2, ... , 100\}$ on $100$ vertices and suppose we are
going to determine all eigenvalues which are not greater than
$\omega=0.002$.  Note that the space $PW_{0.002}(C_{100})$ is the
span of all eigenfunctions whose eigenvalues are not greater than
$0.002$. According to the Theorem 1.3 a uniqueness set for the
space $PW_{0.002}(C_{100})$ can be constructed as a compliment of
a set $S=\bigcup_{j} S_{j}$ such that $\overline{S}_{j}=S_{j}\cup
bS_{j}$ are disjoint and

$$
|S_{j}|< \frac{\pi}{2\arcsin\sqrt{\frac{ 0.002}{2}}}-1>49-1=48.
$$

Thus we can take $|S_{j}|=48$ and it means that one of possible
uniqueness sets $U$ will contain four vertices with numbers 1, 2,
51, and 52. According to the Corollary 1.1 we can conclude that
there are \textit{at most four} eigenvalues of the Laplace
operator which are not greater than $0.002$. In fact \textit{there
are three} such eigenvalues $\lambda_{0}=0,$ and a double
eigenvalue $ \lambda_{1}=1-\cos(2\pi/100)\approx 0.001973$.

Similar calculations show that in the case when $\omega=0.008$ the
dimension of a uniqueness set $U$ can be taken equal eight and
there are five eigenvalues which are less than $0.008$:
$\lambda_{0}=0$ and two double eigenvalues $\lambda_{1}\approx
0.001973,$ and $ \lambda_{2}=1-\cos(4\pi/100)\approx 0.007885.$

The following Corollary gives a lower bound for each non-zero
eigenvalue.

\begin{col}
If $S=\{v_{1},...,v_{N-k}\}$ is a set of $N-k$ vertices and
$\Lambda(S)=\Lambda(v_{1},...,v_{N-k})$ is the corresponding
Poincare constant, then the following inequality holds true
\begin{equation}
\lambda_{k}\geq \frac{1}{\Lambda(v_{1},...,v_{N-k})}.
\end{equation}
More precisely, if
$$\Lambda_{N-k}=\min_{(v_{1},...,v_{N-k})\in
V(G)}\Lambda(v_{1},...,v_{N-k}),
$$
then
$$
\lambda_{k}\geq \frac{1}{\Lambda_{N-k}}.
$$
\end{col}

Note that this result is "local" in the sense that any randomly
chosen set  $S=\{v_{1},...,v_{N-k}\} \subset V(G)$ can be used to
obtain an estimate of the form (1.14).

\section{Paley-Wiener spaces on combinatorial graphs}

The Paley-Wiener spaces $PW_{\omega}(G), \omega>0,$ were
introduced in the Definition 1 of the Introduction. Since the
operator $\mathcal{L}$ is bounded it is clear that every function
from $L_{2}(G)$ belongs to a certain Paley-Wiener space. If a
graph $G$ is finite then the space $PW_{\omega}(G)$ is a span of
eigenfunctions whose eigenvalues $\leq \omega$. Note that if
$$
\omega_{\min}=\inf _{\omega\in \sigma(\mathcal{L})}\omega
$$
then the space $PW_{\omega}(G)$ is not trivial if and only if
$\omega\geq \omega_{\min}$.

 Using
the spectral resolution of identity $P_{\lambda}$ we define the
unitary group of operators by the formula
$$
e^{it\mathcal{L}}f=\int_{\sigma(\mathcal{L})}e^{it\tau}dP_{\tau}f,
f\in L_{2}(G), t\in \mathbb{R}.
$$

 The next theorem can be considered
as a form of the Paley-Wiener theorem and it follows from a more
general result in \cite{Pes00}
\begin{thm}
The following statements hold true:

\begin{enumerate}

\item $f\in PW_{\omega}(G)$ if and only if for all $s\in
\mathbb{R}_{+}$ the following Bernstein inequality holds
\begin{equation}
\|\mathcal{L}^{s}f\|\leq \omega^{s}\|f\|;
\end{equation}

\item  the norm of the operator $\mathcal{L}$ in the space
$PW_{\omega}(G)$ is exactly $\omega$;

\item  $f\in L_{2}(G)$ and  the following limit is finite
$$
\lim_{s\rightarrow \infty}
\|\mathcal{L}^{s}f\|^{1/s}=\omega<\infty, s\in \mathbb{R}_{+},
$$
then $\omega$ is the smallest number for which $f\in
PW_{\omega}(G)$;

\item   $f\in PW_{\omega}(G)$ if and only if for every $g\in
L_{2}(G)$ the scalar-valued function of the real variable $t\in
\mathbb{R}^{1}$
$$
\left<e^{it\mathcal{L}}f,g\right>=\sum_{v\in V}
e^{it\mathcal{L}}f(v)\overline{g(v)}
$$
 is bounded on the real line and has an extension to the complex
plane as an entire function of the exponential type $\omega$;

\item  $f\in PW_{\omega}(G)$ if and only if the abstract-valued
function $e^{it\mathcal{L}}f$  is bounded on the real line and has
an extension to the complex plane as an entire function of the
exponential type $\omega$;

\item  $f\in PW_{\omega}(G)$ if and only if the solution $u(t,v),
t\in \mathbb{R}^{1}, v\in V(G),$ of the Cauchy problem for the
corresponding Schrodinger equation
$$
i\frac{\partial u(t,v)}{\partial t}=\mathcal{L}u(t,v),
u(0,v)=f(v), i=\sqrt{-1},
$$
 has analytic extension $u(z,v)$ to the
complex plane $\mathbb{C}$ as an entire function and satisfies the
estimate
$$
\|u(z, \cdot)\|_{L_{2}(G)}\leq e^{\omega|\Im z|}\|f\|_{L_{2}(G)}.
$$
\end{enumerate}
\end{thm}
  This  Theorem gives the following
stratification
$$
L_{2}(G)=PW_{\omega_{\max}}(G)=\bigcup_{\omega_{\min}\leq
\omega\leq \omega_{\max}}PW_{\omega}(G), PW_{\omega}(G)\subseteq
PW_{\sigma}(G), \omega<\sigma,
$$
of the space of all $L_{2}(G)$-functions. The  Theorem shows that
the notion of Paley-Wiener functions of type $\omega$ on a
combinatorial graph can be completely understood in terms of
familiar  entire functions of exponential type $\omega$ bounded on
the real line.

\section{ Uniqueness sets  for discrete Paley-Wiener functions  and
Plancherel-Polya and Poincare inequalities on graphs}

The definition of uniqueness sets was given in the Introduction.
The following Theorem gives a necessary and sufficient conditions
for being a uniqueness set.
\begin{thm} If $PW_{\omega}(G)$ is finite dimensional for an
$\omega>0$,  then a set of vertices $U\subset V(G)$ is a
uniqueness set for the space $PW_{\omega}(G)$ if and only if
there exists a constant $C_{\omega}$ such that for any $f\in
PW_{\omega}(G)$ the following discrete version of the
Plancherel-Polya inequalities holds true
\begin{equation}
\left(\sum_{u\in U}|f(u)|^{2}\right)^{1/2}\leq\|f\|_{L_{2}(G)}\leq
C_{\omega}\left(\sum_{u\in U}|f(u)|^{2}\right)^{1/2}
\end{equation}
for all $f\in PW_{\omega}(G)$.
\end{thm}

\begin{proof}
The closed linear subspace $PW_{\omega}(G)$ is a Hilbert space
with respect to the norm of $L_{2}(G)$. At the same time since
$U\subset V(G)$ is a uniqueness set for $PW_{\omega}(G)$  the
functional
$$
|||f|||=\left(\sum_{u\in U}|f(u)|^{2}\right)^{1/2}
$$
defines another  norm on $PW_{\omega}(G)$. Indeed, the only
property which should be verified is that the condition
$|||f|||=0, f\in PW_{\omega}(G)$, implies that $f$ is identical
zero on entire graph but it is guaranteed by the fact that $U$ is
a uniqueness set for $PW_{\omega}(G)$.

Since for any $f\in PW_{\omega}(G)$ the norm $|||f|||$ is not
greater than the original norm $\|f\|_{L_{2}(G)}$  the closed
graph Theorem implies the existence of a constant $C_{\omega}$ for
which the reverse inequality holds true.
\end{proof}

We will also need the following Corollary which is easy to prove.
\begin{col}
In the same notations as above if $\mathcal{B}$ is an operator in
$L_{2}(G)$ such that its restriction to $PW_{\omega}(G)$ is a
bounded invertible operator from $PW_{\omega}(G)$ onto
$PW_{\omega}(G)$ then
\begin{equation}
\frac{1}{\|\mathcal{B}\|}\left(\sum_{u\in
U}|\mathcal{B}f(u)|^{2}\right)^{1/2}\leq \|f\|_{L_{2}(G)}\leq
\|\mathcal{B}^{-1}\|C_{\omega}\left(\sum_{u\in
U}|\mathcal{B}f(u)|^{2}\right)^{1/2}
\end{equation}
for all $f\in PW_{\omega}(G)$.
\end{col}

 \begin{rem}
 It is worth to note
that the statement  similar to the Theorem 3.1 does not hold true
for Paley-Wiener spaces on $\mathbb{R}^{d}$. Namely, not every
uniqueness set is associated with a corresponding Plancherel-Polya
inequality. Sets of points $\{x_{j}\}\in \mathbb{R}^{d}$  for
which a Plancherel-Polya inequality takes place   are known as
sampling sets.
\end{rem}

For a general graph $G$ an important example of the operator
$\mathcal{B}$ from the Corollary is the operator $(\varepsilon
I+\mathcal{L})^{s}$ for any positive $\varepsilon
>0$ and any real $s\in \mathbb{R}$.

 Our
next goal is to develop some sufficient conditions on a set of
vertices for being a uniqueness set for a Paley-Wiener subspace.
It turns out that it is easier to understand compliments of
uniqueness sets for spaces $PW_{\omega}(G)$.

\bigskip

We now turn to the notion of a $\Lambda$-set which was introduced
in the Definition 3 in the Introduction. Let us consider one of
our main examples of $\Lambda$-sets.

\begin{exmp}

Let  $v\in V$ be any vertex  in a connected graph $G(V(G), E(G))$.
In this case the corresponding space $L_{2}(\{v\})$ consists of
all functions proportional to the Dirac measure $\delta_{v}$.
Simple calculations show
$$
\mathcal{L}\delta_{v}(v)=1,
\mathcal{L}\delta_{v}(w)=-\frac{1}{\sqrt{d(w)d(v)}}, w\sim v,
$$
and $\mathcal{L}\delta_{v}(u)=0$ for all other vertices $u\in
V(G)$. Thus we have
$$
\|\mathcal{L}f\|=1+\frac{1}{d(v)}\sum_{w\sim v}\frac{1}{d(w)}
$$
and since $\|\delta_{v}\|=1$ we obtain
$$
\|\delta_{v}\|=\Lambda (v)\|\mathcal{L}\delta_{v}\|,
$$
where $\Lambda(v)$  is

\begin{equation}
\frac{1}{\sqrt{2}}\leq\Lambda(v)=\frac{1}{\sqrt{1+\frac{1}{d(v)}\sum_{w\sim
v}\frac{1}{d(w)}}}\leq \sqrt{\frac{d(G)}{1+d(G)}}<1, d(G)\geq 1.
\end{equation}

This example  can be used to construct examples of
\textit{infinite} $\Lambda$-sets. For example, we can consider an
infinite path on a rectangular grid or on a tree. These situations
will be generalized in the Lemma 3.6.

\end{exmp}

The role of $\Lambda$-sets is explained in the following Theorem.

\begin{thm}
If a set $S\subset V(G)$  is a $\Lambda$-set,  then $S$ is
$PW_{\omega}(G)$-removable for any $\omega < 1/\Lambda$ i. e. the
set  $U=V(G)\backslash S$ is a uniqueness set for any space
$PW_{\omega}(G)$ with $ \omega< 1/\Lambda$.
\end{thm}

\begin{proof}
If $f,g\in PW_{\omega}(G)$ then $f-g\in PW_{\omega}(G)$ and
according to the Theorem 2.1 the following Bernstein inequality
holds true
\begin{equation}
\|\mathcal{L}(f-g)\|\leq \omega\|f-g\|.
\end{equation}
If $f$ and $g$ coincide on $U=V(G)\backslash S$ then $f-g$ belongs
to  $L_{2}(S)$ and since $S$ is a $\Lambda$-set then we will have
$$
\|f-g\|\leq \Lambda\|\mathcal{L}(f-g)\|, f-g\in L_{2}(S).
$$
 Thus, if $f-g$ is not zero and $\omega<1/\Lambda$ we have the
following inequalities
\begin{equation}
\|f-g\|\leq \Lambda\|\mathcal{L}(f-g)\|\leq \Lambda
\omega\|f-g\|<\|f-g\|,
\end{equation}
which contradict to the assumption that $f-g$ is not identical
zero. It proves the Theorem.
\end{proof}

This Theorem and the estimate (3.5) show  that for any graph $G$
spaces $PW_{\omega}(G)$ with
\begin{equation}
0<\omega<\sqrt{\frac{1+d(G)}{d(G)}}=\Omega_{G}>1, d(G)=\sup_{v\in
V(G)}d(v),
\end{equation}
have non-trivial uniqueness sets.
\bigskip

Now we show that every finite set with non-empty boundary admits a
Poincare inequality.

\begin{thm} Suppose that $S$ is a finite set with non-empty
boundary.
If $\mathcal{L}_{\Gamma(S)}$ is the Laplacian on the graph
$\Gamma(S)$  and $\lambda_{1}(\Gamma(S))$ is its smallest non-zero
eigenvalue, then for every $\varphi\in L_{2}(S)$ the following
inequality holds true
\begin{equation}
\|\varphi\|_{L_{2}(G)}\leq
\frac{1}{\lambda_{1}(\Gamma(S))}\|\mathcal{L}_{G}\varphi\|_{L_{2}(G)}.
\end{equation}
\end{thm}
\begin{proof}

Let us remind the construction of the graph $\Gamma(S)$. We
consider $\overline{S}=S\cup b S$ as an induced graph,  take two
copies of it which  will be denoted as $\overline{S}_{1}$ and
$\overline{S}_{2}$ and identify every vertex $v\in b S\subset
\overline{S}_{1}$ with "the same" vertex $v\in bS\subset
\overline{S}_{2}$. Now we construct  an embedding  of the space
$L_{2}(S) $ into the space $L_{2}(\Gamma(S))$. If $\varphi\in
L_{2}(S) $ then its image $F_{\varphi} \in L_{2}(\Gamma(S))$ is
defined according to the following rules:
\begin{enumerate}

\item $F_{\varphi}(v)=\varphi(v)$, for every $v\in
\overline{S}_{1}$,

\item $F_{\varphi}(v)=-\varphi(v)$, for every $v\in
\overline{S}_{2}$.

\end{enumerate}

It is clear  the following holds true for every function
$\varphi\in L_{2}(S)$
\begin{equation}
\|F_{\varphi}\|_{L_{2}(\Gamma(S))}=\sqrt{2}\|\varphi\|_{L_{2}(G)}.
\end{equation}
We will use notations $d_{G}(v), d_{\overline{S}}(v),
d_{\Gamma(S)}(v)$ for degrees of a vertex $v$ in $\overline{S}$
considered as a vertex of $G$, or as a vertex of the induced graph
$\overline{S}$ or as a vertex of the new graph $\Gamma(S)$
respectively. It is clear that if $v\in S$ then
$$d_{\Gamma(S)}(v)=
d_{\overline{S}}(v)=d_{G}(v)
$$ and if $v\in bS$ then
$d_{\Gamma(S)}(v)=2d_{\overline{S}}(v)$.

As direct calculations show if $v\in S$ and $\varphi \in L_{2}(G)$
then
$$
\mathcal{L}_{\Gamma(S)}F_{\varphi}=\mathcal{L}_{G}\varphi,\varphi
\in L_{2}(G),
$$
and $$ \mathcal{L}_{\Gamma(S)}F_{\varphi}=0,\varphi \in L_{2}(G).
$$
 Thus we obtain  the following estimate
\begin{equation}\|\mathcal{L}_{\Gamma(S)}F_{\varphi}\|_{L_{2}(\Gamma(S))}\leq
\sqrt{2}\|\mathcal{L}_{G}\varphi\|_{L_{2}(G)}.
\end{equation}
The eigenfunction of $\mathcal{L}_{\Gamma(S)}$ that corresponds to
$0$-eigenvalue is given by the formula $\psi_{0}(v)=\sqrt{d_{
\Gamma(S)}(v)}, v\in \Gamma(S)$. In particular, for $v\in S$ one
has $\psi_{0}(v)=\sqrt{d_{ G}(v)}$ which coincides for $v\in S$
with the harmonic function for  $\mathcal{L}_{G}$.

 Since every function
$F_{\varphi}$ is "odd" it is orthogonal to the subspace spanned by
$\psi_{0}$. Because of it if $\{\psi_{j}\}, j=0,1,..., N,$ is a
complete orthonormal system of eigenfunctions of
$\mathcal{L}_{\Gamma(S)}$ in $L_{2}(\Gamma(S))$ and
$0=\lambda_{0}(\Gamma(S))< \lambda_{1}(\Gamma(S))\leq ...\leq
\lambda_{N}(\Gamma(S)),$ is a set of their corresponding
eigenvalues the following formulas hold true
$$
F_{\varphi}=\sum_{j=1}^{N}
\left<F_{\varphi},\psi_{j}\right>\psi_{j}
$$
and
$$ \mathcal{L}_{\Gamma(S)}F_{\varphi}=\sum_{j=1}^{N}
\lambda_{j}(\Gamma(S))\left<F_{\varphi},\psi_{j}\right>\psi_{j}.
$$
Finally we obtain
$$
\|\mathcal{L}_{\Gamma(S)}F_{\varphi}\|_{L_{2}(\Gamma(S))}^{2}=\sum_{j=1}^{N}
\lambda_{j}^{2}(\Gamma(S))|\left<F_{\varphi},\psi_{j}\right>|^{2}\geq
\lambda_{1}^{2}(\Gamma(S))\|F_{\varphi}\|_{L_{2}(\Gamma(S))}^{2}.
$$
This inequality along with (3.8) and (3.9) imply the Theorem.
\end{proof}

We are going to make use of some known estimates \cite{Ch} on the
first eigenvalue of a finite graph. A graph $\Gamma$ has
\textit{isoperimetric dimension} $\delta$ with
\textit{isoperimetric constant} $c_{\delta}$ if for any subset $W$
of $V(\Gamma)$, the number of edges between $W$ and the complement
$W ^{'}$ of $W$, denoted by $|\partial W|=|E(W, W^{'})|$,
satisfies
$$|E(W, W^{'})|\geq c_{\delta}(vol
W)^{\frac{\delta-1}{\delta}}.
$$
If $\delta$ is the isoperimetric dimension of the graph $\Gamma$
then there exists a constant $C_{\delta}$ which depends just on
$\delta$ such that

$$
\lambda_{1}(\Gamma)>C_{\delta}\left(\frac{1}{vol
\Gamma}\right)^{2/\delta},
$$
were the volume $vol(\Gamma)$
 of a graph $\Gamma$ is defined as
 $$
 vol(\Gamma)=\sum_{v\in V(\Gamma)}d(v).
 $$
From the construction of the graph $\Gamma(S)$ it is clear that
$$
vol(\Gamma(S))=2\left(vol S+vol_{\overline{S}} (b S)\right)\leq
2(vol S+ vol (b S)),
$$
where $vol_{\overline{S}}( b S)$ means the volume of the set $bS$
is calculated under the assumption that this set is a subset if
the $\textit{induced}$ graph $\overline{S}$. There is another
lower estimate of the first non-zero eigenvalue of a graph in
terms of the Cheeger constant $h_{\Gamma}$. To define $h_{\Gamma}$
we define
$$
h_{\Gamma}(W)=\frac{|E(W,W^{'})|}{\min(volW, vol W^{'})}
$$
and then
$$
h_{\Gamma}=\min_{W}h_{\Gamma}(W).
$$
The following lower estimate is known for any connected graph
$\Gamma$
$$
\lambda_{1}(\Gamma)>\frac{h_{\Gamma}^{2}}{2}.
$$
There is a lower estimate on the first eigenvalue in terms of a
diameter and volume  of a graph. Namely, the following estimate
follows from the variational principal \cite{Ch}
 $$
\lambda_{1}(\Gamma)\geq \frac{1}{D(\Gamma)vol(\Gamma)},
$$
where $D(\Gamma)$ is the diameter of the graph $\Gamma$. This
estimate along with the last Theorem gives the following estimate
for $\Gamma=\Gamma(S)$
$$
\|f\|_{L_{2}(G)}\leq
D(\Gamma(S))vol(\Gamma(S))\|\mathcal{L}_{G}f\|_{L_{2}(G)}
$$
for every function from $L_{2}(G)$ with support in $S$. The
construction of the graph $\Gamma(S)$ implies the inequality
$$
D(\Gamma(S))\leq 2D(\overline{S}),
$$
where $\overline{S}=S\bigcup bS$ is considered as the induced
graph.

 All together it gives the following result.
\begin{thm}
If $S\subset V(G)$ is a finite set  with non-empty boundary and
$\Gamma(S)$ is the same as above then for any $\varphi\in
L_{2}(S)$ the following Poincare-type inequalities hold true
$$
\|\varphi\|_{L_{2}(G)}\leq\frac{2}{h_{\Gamma(S)}^{2}}
\|\mathcal{L}_{G}\varphi\|_{L_{2}(G)},
$$

$$
\|\varphi\|_{L_{2}(G)}\leq
D(\Gamma(S))vol(\Gamma(S))\|\mathcal{L}_{G}\varphi\|_{L_{2}(G)},
$$

$$
\|\varphi\|_{L_{2}(G)}\leq 2D(S)\left(vol S+vol_{\overline{S}} (b
S) \right)\|\mathcal{L}_{G}\varphi\|_{L_{2}(G)}.
$$

There exists a  $C_{\delta}>0$ which depends just on the
isoperimetric constant $\delta$ of the graph $\Gamma(S)$ such that
\begin{equation}
\|\varphi\|_{L_{2}(G)}\leq C_{\delta}^{-1}\left(vol
\Gamma(S)\right)^{2/\delta}\|\mathcal{L}_{G}\varphi\|_{L_{2}(G)},
\end{equation}
$$
\|\varphi\|_{L_{2}(G)}\leq C_{\delta}^{-1}\left(2(vol S+
vol_{\overline{S}} (b
S))\right)^{2/\delta}\|\mathcal{L}_{G}\varphi\|_{L_{2}(G)}.
$$
\end{thm}
Note that since the isoperimetric dimension $\delta$ is a
generalization of such notion as the dimension of a manifold, the
estimate (3.10) is a generalization of the following estimate for
the classical Laplace operator $\Delta$ on a compact domain
$S\subset \mathbb{R}^{d}$
\begin{equation}
\|\varphi\|_{L_{2}(\mathbb{R}^{d})}\leq C_{d}(diam S)^{2}\|\Delta
\varphi\|_{L_{2}(\mathbb{R}^{d})}, \varphi \in C_{0}^{\infty}(S).
\end{equation}

In the case of $\mathbb{R}^{d}$ such kind inequalities play an
important role  in harmonic analysis and differential equations
and usually associated with the names of Wirtinger, Poincare, and
Sobolev \cite{FTT}, \cite{HLP}, \cite{T2}.

 The following two Lemmas  describe some \textit{infinite}
 $\Lambda$-sets. The first Lemma is obvious.
\begin{lem}
Suppose that   a set of vertices $S\subset V(G)$ (finite or
infinite) has the property that for any $v\in S$ its closure
$\overline{v}=v\cup bv$ does not contain other points of $S$, then
$S$ is a $\Lambda$-set with $\Lambda=1$, i. e.
$$
\|\varphi\|\leq \|\mathcal{L}\varphi\|, \varphi\in L_{2}(S).
$$
\end{lem}
\begin{lem}
Suppose that for  a set of vertices $S\subset V(G)$ (finite or
infinite) the following conditions hold true:

\begin{enumerate}

\item every point from $S$ is adjacent to a point from the
boundary;

\item for every $v\in S$ there exists at least one adjacent point
$u_{v}\in b S$ whose adjacency set intersects $S$ only over $v$;

\item the number
\begin{equation}
\Lambda=\sup_{v\in S}d(v)
\end{equation}
is finite.
\end{enumerate}
 Then the set $S$ is a $\Lambda$-set.
\end{lem}
\begin{proof}
Assumptions of the Lemma imply that there exists  a subset
$S^{*}\subset b S$ such that for every vertex $v\in S$ there
exists at least one point $u_{v}\in S^{*}$ whose adjacency set
intersects $S$ only over $v$. A direct calculation shows that if
$\varphi \in L_{2}(S)$ then
\begin{equation}
\mathcal{L}\varphi(u_{v})=-\frac{\varphi(v)}{\sqrt{d(v)d(u_{v})}},u_{v}\in
S^{*}, v\in S.
\end{equation}

Define $\Lambda$ by the formula (3.12). Since by assumption for
every $v\in S$ there exists at least  one vertex $u_{v}$ in
$S^{*}$ which is adjacent to $v$ we obtain
\begin{equation}
\|\mathcal{L} \varphi\|\geq \left(\sum_{v\in S}\left|\mathcal{L}
\varphi(u_{v})\right|^{2}\right)^{1/2}\geq
\Lambda^{-1}\|\varphi\|.
\end{equation}
 In particular if
degrees of all vertices in $S$ and $S^{*}$ are uniformly bounded
from above by a number $d(G)$, then $\Lambda\leq d(G)$. The Lemma
is proved.
\end{proof}

 The following property is
important and allows to construct infinite $\Lambda$-sets from the
finite ones.

\begin{lem}
Suppose that $\{S_{j}\}$ is a finite or infinite sequence of
disjoint subsets of vertices  $S_{j}\subset V$ such that the sets
$S_{j}\cup bS_{j}$ are pairwise disjoint. Then if a set $S_{j}$
has type $\Lambda_{j}$, then their union $S=\bigcup_{j} S_{j}$ is
a set of type $\Lambda=\sup_{j} \Lambda_{j}$.
\end{lem}
\begin{proof}
Since the sets $S_{j}$ are disjoint every function $\varphi\in
L_{2}(S), S=\bigcup_{j} S_{j}$, is a sum of functions
$\varphi_{j}\in L_{2}(S_{j})$ which are pairwise  orthogonal.
Moreover because the sets $S_{j}\cup bS_{j}$ are disjoint the
functions $\mathcal{L}\varphi_{j}$ are also orthogonal. Thus we
have
$$
\|\varphi\|^{2}=\sum_{j}\|\varphi_{j}\|^{2}\leq
\sum_{j}\Lambda_{j}^{2}\|\mathcal{L}\varphi_{j}\|^{2}\leq
\Lambda^{2}\|\mathcal{L}\varphi\|^{2},
$$
where $\Lambda=\sup_{j}\Lambda_{j}$. The Lemma is proved.
\end{proof}

A combination of the last Lemma and the Theorem 3.3 gives the
following uniqueness result.

\begin{thm}
For a given $\omega_{\min}<\omega< \Omega_{G}$ consider a set
$S=\bigcup_{j} S_{j}$ with the following properties:

\begin{enumerate}

\item  every $S_{j}$ is a finite set with non-empty boundary and
the following inequality takes place
$$
\lambda_{1}(\Gamma(S_{j}))>\omega,
$$
\item  the sets $S_{j}\cup bS_{j}$ are disjoint.

\end{enumerate}
 Then
the set $U=V(G)\setminus S$ is a uniqueness set for the space
$PW_{\omega}(G)$.
\end{thm}

Although the space $L_{2}(S)$ is not invariant under $\mathcal{L}$
the inequality (1.7) implies infinitely many similar inequalities.
Namely we have the following result which will be used later.
\begin{lem}
If $S$ is a $\Lambda$-set, then for any $\varphi\in L_{2}(S)$ and
all $t\geq 0, k=2^{l}, l=0,1,2,...$
$$
\|\mathcal{L}^{t}\varphi\|\leq
\Lambda^{k}\|\mathcal{L}^{k+t}\varphi\|, \varphi\in L_{2}(S),
 $$
 in particular
$$
\|\varphi\|\leq \Lambda^{k}\|\mathcal{L}^{k}\varphi\|, \varphi\in
L_{2}(S).
$$
\end{lem}

\begin{proof}
By the spectral theory \cite{BS} there exist a direct integral of
Hilbert spaces $X=\int X(\tau)dm (\tau )$ and a unitary operator
$F$ from $L_{2}(G)$ onto $X$, which transforms domain of
$\mathcal{L}^{t}, t\geq 0,$ onto $X_{t}=\{x\in X|\tau^{t}x\in X
\}$ with norm
$$
\|\mathcal{L}^{t}f\|_{L_{2}(G)}=\left (\int_{\mathbb{R}_{+}}
\tau^{2t} \|Ff(\tau )\|^{2}_{X(\tau)} d m(\tau) \right )^{1/2}
$$
and $F(\mathcal{L}^{t} f)=\tau ^{t} (Ff)$. For any $\varphi\in
L_{2}(G)$  we have
$$
\int _{\mathbb{R}_{+}}| F\varphi(\tau)|^{2}d m(\tau)\leq
\Lambda^{2} \int _{\mathbb{R}_{+}}\tau ^{2}| F\varphi(\tau)|^{2}d
m(\tau)
$$
and then for the ball $B=B(0, \Lambda^{-1})$ we have
$$\int _{B}| F\varphi(\tau)|^{2}d m(\tau)+
\int_{\mathbb{R}_{+}\setminus B}|F\varphi|^{2}d m(\tau)\leq
$$
$$
\Lambda^{2}\left( \int_{B}\tau^{2}|F\varphi|^{2}d m(\tau)
+\int_{\mathbb{R}_{+}\setminus B}\tau^{2}|F\varphi| ^{2}d
m(\tau)\right ) .
$$
 Since $\Lambda^{2}\tau^{2}<1$ on $B(0, \Lambda^{-1})$
$$
0\leq \int_{B}\left
(|F\varphi|^{2}-\Lambda^{2}\tau^{2}|F\varphi|^{2}\right)d m(\tau)
\leq \int _{\mathbb{R}_{+}\setminus B}\left ( \Lambda^{2}
\tau^{2}|F\varphi|^{2}-|F\varphi|^{2}\right)d m(\tau).
$$
This inequality
 implies the inequality
 $$
 0\leq \int_{B}\left (\Lambda^{2}\tau^{2}|F\varphi|^{2}-
 \Lambda^{4}\tau^{4}|F\varphi|^{2}\right)d m(\tau)
\leq \int_{\mathbb{R}_{+}\setminus B}\left ( \Lambda^{4}
\tau^{4}|F\varphi|^{2}-\Lambda^{2}\tau^{2}|F\varphi| ^{2}\right)d
m(\tau)$$
 or
  $$\Lambda^{2}\int_{\mathbb{R}_{+}}\tau^{2}|F\varphi|^{2}d m(\tau) \leq
  \Lambda^{4}\int_{\mathbb{R}_{+}}\tau^{4}|F\varphi|
 ^{2}d m(\tau),
 $$
 and then
 $$
 \|\varphi\|\leq \Lambda\|\mathcal{L}\varphi\|\leq \Lambda^{2}\|\mathcal{L}^{2}\varphi\|
 , \varphi\in L_{2}(S).
 $$

 Next by using induction one can show the inequality
 $\|\varphi\|\leq \Lambda^{k}\|\mathcal{L} ^{k}\varphi\|$
 for any $k=2^{l}, l=0, 1, ... .$ Then again, because for any
 $t\geq 0, \Lambda^{2t}\tau^{2t}<1$ on
 $B(0, \Lambda^{-1})$ we have
$$0\leq \int_{B}\left (\Lambda^{2t}\tau^{2t}|F\varphi|^{2}-
\Lambda^{2(k+t)}\tau^{2(k+t)}|F\varphi|^ {2}\right)d m(\tau) \leq
$$
$$
\int_{\mathbb{R}_{+}\setminus B}\left ( \Lambda^{2(k+t)}
\tau^{2(k+t)}|F\varphi|^{2}- \Lambda^{2t}\tau^ {2t}|F\varphi|
^{2}\right)d m(\tau),
$$
that gives $\|\mathcal{L}^{t}\varphi\|\leq
\Lambda^{k}\|\mathcal{L} ^{(k+t)}\varphi\|,  \varphi\in L_{2}(S).$
Lemma is proved.
\end{proof}

\section{A reconstruction algorithm in terms of dual frames in Hilbert spaces}

In this section we will use the Theorem 3.1 to develop a
reconstruction method in terms of Hilbert frames. Recall that a
set of vectors $\{h_{j}\}$ from a Hilbert space $H$ is called a
frame in $H$ if there are $0<A<B$ such that for any $f\in H$
$$
A\sum_{j}|\left<f,h_{j}\right>_{H}|^{2}\leq\|f\|_{H}^{2}\leq
B\sum_{j}|\left<f,h_{j}\right>_{H}|^{2},
$$
where $\left<.,.\right>_{H}$ is the inner product in $H$. The
ratio $A/B$ is called the tightness of the frame.

 Let
$\delta_{v}\in L_{2}(G)$ be a Dirac measure supported at a vertex
$v\in V$. The notation $\vartheta_{v}$ will be used for a function
which is orthogonal projection of the function
$$
\frac{1}{\sqrt{d(v)}}\delta_{v}
$$
on the subspace $PW_{\omega}(G)$. Then the Plancherel-Polya
inequalities (3.1) can be written in the form
\begin{equation}
\sum_{u\in
U}|\left<f,\vartheta_{u}\right>|^{2}\leq\|f\|_{L_{2}(G)}^{2}\leq
C_{\omega}^{2}\sum_{u\in U}|\left<f,\vartheta_{u}\right>|^{2},
\end{equation}
where $f, \vartheta_{u}\in PW_{\omega}(G)$ and
$\left<f,\vartheta_{u}\right>$ is the inner product in $L_{2}(G)$.
These inequalities mean that if $U$ is a uniqueness set for the
subspace $PW_{\omega}(G)$ then the functions
$\{\vartheta_{u}\}_{u\in U}$ form a frame in the subspace
$PW_{\omega}(G)$ and the tightness of this frame is
$1/C_{\omega}^{2}$. Following an idea of Duffin and Schaeffer
\cite{DS} we sketch the proof of the following Theorem which gives
a reconstruction formula similar to the formulas (1.1) and (1.4).
\begin{thm}
If $U\subset V(G)$ is a uniqueness set for the subspace
$PW_{\omega}(G)$ then there exists a    frame of functions
$\{\Theta_{u}\}_{u\in U}$ in the space $PW_{\omega}(G)$ such that
the following reconstruction formula holds true for all $f\in
PW_{\omega}(G)$
\begin{equation}
f(v)=\sum_{u\in U}f(u)\Theta_{u}(v), v\in V(G).
\end{equation}
\end{thm}
\begin{proof}
  The  idea is to show that the
so-called frame operator
\begin{equation}
Ff=\sum_{u\in U}\left<f,\vartheta_{u}\right>\vartheta_{u},\>\>\>\>
f\in PW_{\omega}(G),
\end{equation}
is an automorphism of the space $PW_{\omega}(G)$ onto itself and $
 \|F\|\leq C_{\omega},\>\>\>\> \|F^{-1}\|\leq 1.$
We consider an increasing sequence of finite subsets of $U$
$$
U_{1}\subset U_{2}\subset...\subset U
$$
and introduce the operator $ F_{j}:PW_{\omega}(G)\rightarrow
PW_{\omega}(G),$ which is given by the formula
$$
F_{j}f=\sum_{u\in
U_{j}}\left<f,\vartheta_{u}\right>\vartheta_{u},\>\>\>\> f\in
PW_{\omega}(G).
$$
It can be shown that the Plancherel-Polya inequalities (4.1) imply
that the limit
$$
\lim _{j\rightarrow \infty}F_{j}f=Ff,\>\>\>\> f\in PW_{\omega}(G),
$$
 exists. We also have
$$
\|Ff\|^{2}=\sup_{\|h\|=1} \left|\sum_{u\in
U}\left<f,\vartheta_{u}\right>\left<\vartheta_{u},h\right>
\right|^{2}\leq \sup_{\|h\|=1}C_{\omega}^{2}\|f\|^{2}\|h\|^{2}
=C_{\omega}^{2}\|f\|^{2},
$$
which shows that the operator $F$ is continuous. The same
Plancherel-Polya  inequalities (4.1) imply that $ I\leq F\leq
C_{\omega}I,$ where $I$ is the identity operator. Thus, we have
$$
0\leq I-C_{\omega}^{-1}F\leq
I-C_{\omega}^{-1}I=\left(C_{\omega}-1\right)C_{\omega}^{-1}I,
$$
and then
$$
\left\|I-C_{\omega}^{-1}F\right\|\leq
\left\|\left(C_{\omega}-1\right)C_{\omega}^{-1}I\right\|\leq
\left(C_{\omega}-1\right)C_{\omega}^{-1}<1.
$$
It shows that the operator $\left(C_{\omega}^{-1}F\right)^{-1}$
and consequently the operator $F^{-1}$ are  bounded operators and
the Neumann series gives the desired estimate $\|F^{-1}\|\leq 1$:

$$
F^{-1}=C_{\omega}^{-1}\left(C_{\omega}^{-1}F\right)^{-1}=C_{\omega}^{-1}
\sum_{k=0}^{\infty}\left(I-C_{\omega}^{-1}F\right)^{k}.
$$
Thus we have
$$
f=F^{-1}Ff=\sum_{u\in U}\left<f,\vartheta_{u}\right>\Theta_{u}
$$
where the functions $ \Theta_{u}=F^{-1}\vartheta_{u}, $ form a
dual frame $\{\Theta_{u}\}_{u\in U}$ in the space
$PW_{\omega}(G)$. The Theorem is proved.

\end{proof}

In a similar way we can prove the following result about a
"derivative sampling".
\begin{thm}
If $U\subset V(G)$ is a uniqueness set for the subspace
$PW_{\omega}(G)$ then for any $s\in \mathbb{R}$ there exists a
frame of functions $\{\Phi_{u}^{(s)}\}_{u\in U}$ in the space
$PW_{\omega}(G)$ such that the following reconstruction formula
holds true for all $f\in PW_{\omega}(G)$
\begin{equation}
f(v)=\sum_{u\in U}(I+\mathcal{L})^{s}f(u)\Phi_{u}^{(s)}(v), v\in
V(G).
\end{equation}
\end{thm}

If $U\subset V(G)$ is a uniqueness set for a space
$PW_{\omega}(G)$ the notation $l_{2,\omega}(U)$ will be used for a
linear subspace of all sequences $a=\{a_{u}\}, u\in U$ in $l_{2}$
for which there exists a function $f$ in $PW_{\omega}(G)$ such
that
$$
f(u)=a_{u}, u\in U.
$$

In general $l_{2,\omega}(U)\neq l_{2}$. A linear reconstruction
method $\mathcal{R}$  is a linear operator $
\mathcal{R}:l_{2,\omega}(U)\rightarrow PW_{\omega}(G) $ such that
$$
\mathcal{R}:y \rightarrow f, y=\{y_{u}\}, y_{u}=f(u), u\in U.
$$
 The reconstruction method $\mathcal{R}$ is said to be stable, if it is
continuous. We obviously have the following statement about stable
reconstruction from derivatives.

\begin{col}
 For any uniqueness set $U $ and any $s\in \mathbb{R}$ the
  reconstruction of $f$ from the
corresponding set of samples
$\left\{(I+\mathcal{L})^{s}f(u)\right\}, u\in U,$ is stable.
\end{col}

\section{Lattice $\mathbb{Z}^{n}$ and homogeneous trees}

We consider a one-dimensional lattice $\mathbb{Z}$.  The dual
group of the commutative additive group $\mathbb{Z}$ is the
one-dimensional torus. The corresponding Fourier transform
$\mathcal{F}$ on the space $L_{2}(\mathbb{Z})$   is defined by the
formula
$$
\mathcal{F}(f)(\xi)=\sum_{k\in \mathbb{Z}}f(k)e^{ik\xi}, f\in
L_{2}(\mathbb{Z}), \xi\in [-\pi, \pi).
$$
It gives a unitary operator from $L_{2}(\mathbb{Z})$ on the space
$L_{2}(\mathbb{T})=L_{2}(\mathbb{T}, d\xi/2\pi),$ where
$\mathbb{T}$ is the one-dimensional torus and $d\xi/2\pi$ is the
normalized measure.  One can verify the following formula
$$
\mathcal{F}(\mathcal{L}f)(\xi)=2\sin^{2}\frac{\xi}{2}\mathcal{F}(f)(\xi).
$$

The next result is obvious.
\begin{thm} The spectrum of the Laplace operator $\mathcal{L}$ on the one-dimensional
lattice $\mathbb{Z}$ is the set $[0, 2]$. A function $f$ belongs
to the space $PW_{\omega}(\mathbb{Z}), 0\leq \omega\leq 2,$ if and
only if the support of $\mathcal{F}f$ is a subset
$\Omega_{\omega}$ of $[-\pi,\pi)$ on which
$2\sin^{2}\frac{\xi}{2}\leq \omega$.

\end{thm}

To formulate the next Theorem we introduce the restriction
operator
$$
\mathcal{R}_{\omega}^{S}: PW_{\omega}(\mathbb{Z})\rightarrow
L_{2}(S), S\subset V(\mathbb{Z}),
$$
where
$$
\mathcal{R}_{\omega}^{S}(\varphi)=\varphi |_{S}, \varphi \in
PW_{\omega}(\mathbb{Z}).
$$
\begin{thm}
For any finite set of vertices $S\subset \mathbb{Z}$ and for every
$\omega>0$  the restriction operator $\mathcal{R}_{\omega}^{S}$ is
surjective.

\end{thm}
\begin{proof}
Assume that a function $h\in L_{2}(S)$  is orthogonal to
restrictions  of all functions from a space
$PW_{\omega}(\mathbb{Z}), 0<\omega<2$ to a finite  $S$. Since $h$
has support on a finite set $S$ its Fourier transform
$\mathcal{F}$ is a finite combination of exponents and in
particular an analytic function on the real line. It implies that
the set of zeros of $\mathcal{F}h$ has measure zero. At the same
time by the Parseval's relation this function $\mathcal{F}h$
should be orthogonal to any function with support in the set
$\Omega_{\omega}$ which is a  set of  positive measure. This
contradiction shows that there is no function in $L_{2}(S)$ which
is orthogonal to restrictions to $S$ of all functions from
$PW_{\omega}(\mathbb{Z})$. Since the set $S$ is finite the space
$L_{2}(S)$ is finite dimensional and it implies that  $L_{2}(S)$
is exactly the set of all restrictions of
$PW_{\omega}(\mathbb{Z})$ to $S$. The Theorem is proved.
\end{proof}

 Our
nearest goal is to show that for a one-dimensional line graph
$\mathbb{Z}$  the estimates in Poincare inequalities of finite
sets can be improved and all the constant can be computed
explicitly. What follows is a specific realization of the
construction and of the proof of the Theorem 3.3.

Consider a set of successive vertices $S=\{v_{1},
v_{2},...,v_{N}\}\subset \mathbb{Z}, $ and the corresponding space
$L_{2}(S)$. If $bS=\{v_{0},v_{N+1}\}$ is the boundary of $S$, then
for any $\varphi\in L_{2}(S)$ the function
$\mathcal{L}_{\mathbb{Z}}\varphi$ has support on $S\cup b S$ and
$$
\mathcal{L}_{\mathbb{Z}}\varphi(v_{0})=-\varphi(v_{1}),
\mathcal{L}_{\mathbb{Z}}\varphi(v_{1})=2\varphi(v_{1})-\varphi(v_{2}),
$$
$$
\mathcal{L}_{\mathbb{Z}}\varphi(v_{N})=2\varphi(v_{N})-\varphi(v_{N-1}),
\mathcal{L}_{\mathbb{Z}}\varphi(v_{N+1})=-\varphi(v_{N}),
$$
and for any other $v_{j}$ with $2\leq j\leq N-1,$
$$
\mathcal{L}_{\mathbb{Z}}\varphi(v_{j})=-\varphi(v_{N-1})+2\varphi(v_{j})-\varphi(v_{N+1}).
$$
 Let $C_{2N+2}=\Gamma(S)$ be a cycle graph
$$
C_{2N+2}=\{u_{-N-1},u_{-N},...,u_{-1}, u_{0},
u_{1},u_{2},...,u_{N}, u_{N+1}\}
$$
with  the following identification
$$
u_{-N-1}= u_{N+1}.
$$
Thus the total number of vertices in $C_{2N+2}$ is $2N+2$. We
introduce an embedding of $S\cup bS$ into $C_{2N+2}$ by the
following identification
$$v_{_{0}}=u_{0},
v_{1}=u_{1},...,v_{N}=u_{N},v_{N+1}=u_{N+1}.
$$
This embedding gives a rise to an embedding of $L_{2}(S)$ into
$L_{2}(C_{2N+2})$, namely every $\varphi\in L_{2}(S)$ is
identified with a function $F_{\varphi}\in L_{2}(C_{2N+2})$ for
which
$$
F_{\varphi}(u_{0})=0, F_{f\varphi}(u_{1})=\varphi(v_{1}), ...,
F_{\varphi}(u_{N})=\varphi(v_{N}), F_{\varphi}(u_{N+1})=0,
$$
and also
$$
 F_{\varphi}(u_{-1})=-\varphi(v_{1}), ..., F_{\varphi}(u_{-N})=-\varphi(v_{N}).
$$
It is important to note  that
$$
\sum_{u\in C_{2N+2}}F_{\varphi}(u)=0.
$$
If  $\mathcal{L}_{C}$ is the Laplace operator on the cycle
$C_{2N+2}$ then a direct computation  shows that for the vector
$F_{\varphi}$ defined above  the following is true
$$
2\|\varphi\|=\|F_{\varphi}\|,
2\|\mathcal{L}_{\mathbb{Z}}\varphi\|=\|\mathcal{L}_{C}F_{\varphi}\|,\varphi\in
L_{2}(S),F_{\varphi}\in L_{2}(C_{2N+2}).
$$
The operator $\mathcal{L}_{C}$ in $ L_{2}(C_{2N+2})$ has a
complete system of orthonormal eigenfunctions
\begin{equation}
\psi_{n}(k) =\exp 2\pi i\frac{n}{2N+2}k, 0\leq n\leq 2N+1,1\leq
k\leq 2N+2,
\end{equation}
with eigenvalues
\begin{equation}
\lambda_{n}=1-\cos \frac{2\pi n}{2N+2}, 0\leq n\leq 2N+1.
\end{equation}
The definition of  the function $F_{\varphi}\in L_{2}(C_{2N+2})$
implies that it  is orthogonal to all constants and its Fourier
series does not contain a term which corresponds to the index
$n=0$. It allows to obtain the following estimate
$$
\|\mathcal{L}_{C}F_{\varphi}\|^{2}=\sum
_{n=1}^{2N+1}\lambda_{n}^{2}\left|\left<F_{\varphi},\psi_{n}\right>\right|^{2}\geq
4\sin^{4}\frac{\pi}{2N+2}\|F_{\varphi}\|^{2}.
$$

It gives the following  estimate for functions $f$ from $L_{2}(S)$
$$
\|\varphi\|\leq \frac{1}{2}\sin^{-2}\frac {\pi}{2N+2}
\|\mathcal{L}_{\mathbb{Z}}\varphi\|
$$

Thus we  proved the following Lemma.

\begin{lem}
If $S=\{v_{1},v_{2},...,v_{N}\}$ consists of $|S|=N$ successive
vertices of a line graph $\mathbb{Z}$ then it is a $\Lambda$-set
for
$$
\Lambda=\frac{1}{2}\sin^{-2}\frac{\pi}{2|S|+2}.
$$
In other words, for any $\varphi\in L_{2}(S)$ the following
inequality holds true
$$
\|\varphi\|\leq \Lambda\|\mathcal{L}_{\mathbb{Z}}\varphi\|.
$$
\end{lem}

Note that in the case $|S|=1$ the last Lemma gives the inequality
$$
\|\delta_{v}\|\leq \|\mathcal{L}_{\mathbb{Z}}\delta_{v}\|,
S=\{v\},
$$
but direct calculations in the Example 1 give a better value for
$\lambda$:
$$
\|\delta_{v}\|=\sqrt{\frac{2}{3}}\left\|\mathcal{L}_{\mathbb{Z}}\delta_{v}\right\|,
v\in V.
$$

 A combination
of this Lemma with Lemma 3.7 and Theorem 3.2 gives the Theorem 1.4
from the Introduction.

\bigskip

 A similar result holds true for a lattice
$\mathbb{Z}^{n}$ of any dimension. Consider for example the case
$n=2$. In this situation  the Fourier transform $\mathcal{F}$ on
the space $L_{2}(\mathbb{Z}^{2})$ is the  unitary operator
$\mathcal{F}$ which is defined by the formula
$$
\mathcal{F}(f)(\xi_{1}, \xi_{2})=\sum_{(k_{1},k_{2})\in
\mathbb{Z}^{2}}f(k_{1}, k_{2})e^{i k_{1}\xi_{1}+i k_{2}\xi_{2}},
f\in L_{2}(\mathbb{Z}\times \mathbb{Z}),
$$
where $ (\xi_{1},\xi_{2})\in [-\pi, \pi)\times [-\pi, \pi)$. The
operator $\mathcal{F}$ is isomorphism  of the space $L_{2}(G)$ on
the space $L_{2}(\mathbb{T}\times
\mathbb{T})=L_{2}(\mathbb{T}\times \mathbb{T},
d\xi_{1}d\xi_{2}/4\pi^{2}),$ where $\mathbb{T}$ is the
one-dimensional torus. the following formula holds true
$$
\mathcal{F}(\mathcal{L}f)(\xi)=\left(\sin^{2}\frac{\xi_{1}}{2}+
\sin^{2}\frac{\xi_{2}}{2}\right)\mathcal{F}(f)(\xi).
$$
We have the following result.
\begin{thm} The spectrum of the Laplace operator on the lattice $\mathbb{Z}^{2}$
 is the set $[0,2]$.  A function $f$ belongs to the space
$PW_{\omega}(\mathbb{Z}^{2}), 0\leq \omega\leq 2,$ if and only if
the support of $\mathcal{F}f$ is a subset $\Omega_{\omega}$ of
$[-\pi,\pi)\times [-\pi,\pi)$ on which
$$
\sin^{2}\frac{\xi_{1}}{2}+ \sin^{2}\frac{\xi_{2}}{2}\leq \omega.
$$
\end{thm}

The same proof as in the case of one-dimensional lattice gives the
following Theorem.

\begin{thm}
If graph $G$ is the $2$-dimensional lattice $\mathbb{Z}^{2}$, then
for any finite set of vertices $S\subset \mathbb{Z}^{2}$, any
$\varphi\in L_{2}(S)$ and any $0\leq \omega\leq 2$ there exists a
function $f_{\varphi}\in PW_{\omega}(\mathbb{Z}^{2})$ which
coincide with $\varphi$ on $S$.
\end{thm}
 Given a set $S=\{v_{n,m}\}, 1\leq n\leq N, 1\leq
m\leq M,$ we consider embedding of $S$ into two-dimensional
discrete torus of the size $T=(2N+2)\times (2M+2)=\{u_{n,m}\}$.
Every $f\in L_{2}(S)$ is identified with a function $g\in
L_{2}(T)$ in the following way
$$
g(u_{n,m})=f(v_{n,m}),1\leq n\leq N, 1\leq m\leq M,
$$
and
$$
g(u_{n,m})=0, N<n\leq N+2, M<m\leq M+2.
$$

We have
$$
\|\mathcal{L}_{G}f\|=\|\mathcal{L}_{T}g\|
$$
where $\mathcal{L}_{T}$ is the combinatorial Laplacian on the
discrete torus $T$. Since eigenfunctions of $\mathcal{L}_{T}$ are
products of the corresponding functions (5.1) a  direct
calculation gives the following inequality
$$
\|\varphi\|\leq \frac{1}{4}\frac{1}{\min\left(\sin
\frac{\pi}{2N+2},\sin
\frac{\pi}{2M+2}\right)}\|\mathcal{L}_{G}\varphi\|, \varphi\in
L_{2}(S).
$$

In a similar way one can  obtain corresponding results for a
lattice $\mathbb{Z}^{n}$ of any dimension. Note that the spectrum
of the Laplace operator on $\mathbb{Z}^{n}$ is $[0, 2]$ and
$\Omega_{\mathbb{Z}^{n}}=\sqrt{(2n+1)/2n}$.

Let $N_{j}=\{N_{1,j},...,N_{n,j}\}, j\in \mathbb{N}, $  be a
sequence $n$-tuples of natural numbers.  For every $j$ the
notation $S(N_{j})$ will be used for a "rectangular solid" of
"dimensions" $N_{1,j}\times N_{2,j}\times...\times N_{n,j}$.

Using these notations we  formulate the following sampling
Theorem.

\begin{thm}

If $S$ is a finite or infinite union of rectangular solids
$\{S(N_{j})\}$
    of vertices of dimensions  $N_{1,j}\times N_{2,j}\times...\times N_{n,j}$
    such that
\begin{enumerate}

 \item  the sets $\overline{S}_{j}=S(N_{j})\cup b S(N_{j})$ are disjoint,

and

  \item  the following inequality holds true for all $j$
$$
\omega< 4\min\left(\sin \frac{\pi}{2N_{1,j}+2},\sin
\frac{\pi}{2N_{2,j}+2},...,\sin \frac{\pi}{2N_{n,j}+2}\right),
$$
then every $f\in PW_{\omega}(\mathbb{Z}^{n})$ is uniquely
determined by its values on $U=V(\mathbb{Z}^{n})\backslash S$.

\end{enumerate}
\end{thm}

As a consequence we obtain the Theorem 1.5 from the Introduction.

\bigskip

We turn now to homogeneous trees. On homogeneous trees there is a
well developed harmonic analysis \cite{FT}, \cite{CMS}. In
particular there is the  Helgason-Fourier transform which provides
the spectral resolution of the combinatorial Laplacian
$\mathcal{L}$. One can use this Helgason-Fourier transform to give
more explicit definition of the Paley-Wiener spaces
$PW_{\omega}(G)$.

We briefly introduce the spherical harmonic analysis on
homogeneous trees. Let $G$ be a homogeneous tree of order $q+1,
q\geq 2$ and $o$ be a root of it. The distance of $v\in V(G)$ from
$o$ is denoted by $|v|$. A function in $L_{2}(G)$ is said to be
radial if it depends only on $|v|$. The space of all radial
functions in  $L_{2}(G)$ will be denoted by $L_{2}^{\natural}(G)$
and the space of all radial functions in $PW_{\omega}(G)$ will be
denoted by $PW_{\omega}^{\natural}(G)$.

We introduce the notation $\tau=2\pi/\ln q$ and consider the torus
$\mathbb{T}=\mathbb{R}/\tau\mathbb{Z}$ which is identified with
$[-\tau/2, \tau/2)$. Let $\mu$ denote the Plancherel measure on
$\mathbb{T}$, given by the formula
$$
d\mu(\xi)=\frac{q\ln q}{4\pi(q+1)}|\textbf{c}(\xi)|^{-2}d\xi,
$$
where
$$
\textbf{c}(\xi)=\frac{q^{1/2}}{q+1}\frac{q^{1/2+i\xi}-q^{-1/2-i\xi}}{q^{i\xi}-q^{-i\xi}},
\xi\in \mathbb{T}.
$$

The spherical functions $\Phi_{\xi}(v), \xi\in[-\tau/2, \tau/2),
v\in V(G)$, are the radial eigenfunctions of the Laplace operator
$\mathcal{L}$ satisfying $\Phi_{\xi}(v)(o)=1.$ The explicit
formula for such functions is
$$
\textbf{c}(\xi)q^{(i\xi-1/2)|v|}+\textbf{c}(\xi)q^{(-i\xi-1/2)|v|},
\xi\in [-\tau/2, \tau/2), v\in V(G).
$$
and the corresponding eigenvalue is given by the formula
\begin{equation}
1-\eta(q)\cos (\xi \ln q), \eta(q)=\frac{2q^{1/2}}{q+1}.
\end{equation}

The spherical Helgason-Fourier transform is defined by the formula
$$
\mathcal{H}f(\xi)=\sum_{v\in V(G)}f(v)\Phi_{\xi}(v), \xi\in
[-\tau/2, \tau/2), f\in L_{1}^{\natural}(G).
$$
The following Theorem is known \cite{FT}
\begin{thm}
The spherical Helgason-Fourier transform extends to an isometry of
$L_{2}^{\natural}(G)$ onto $L_{2}(\mathbb{T}, \mu)$, and
corresponding Plancherel formula holds
$$
\|f\|=\left(\int_{-\tau/2}^{\tau/2}\left|\mathcal{H}(f)(\xi)\right|^{2}d\mu(\xi)\right)^{1/2},
f\in L_{2}^{\natural}(G).
$$
\end{thm}
 Using the fact that $\Phi_{z}(v), v\in V(G),$ is an
eigenfunction of $\mathcal{L}$ with the eigenvalue (5.3) one can
obtain the following formula
$$
\mathcal{H}(\mathcal{L}f)(\xi)=\left(1-\eta(q)\cos (\xi \ln
q)\right)\mathcal{H}(f)(\xi).
$$
The next  statement is obvious.

\begin{thm}
If $G$ is a homogeneous tree of order $q+1$ then a  function $f$
belongs to the space $ PW_{\omega}^{\natural}(G)$ if and only if
the support of $\mathcal{H}f$ is a subset $\Pi_{\omega}$ where
$$\Pi_{\omega}=\left\{\xi\in [1-\eta(q),1+\eta(q)]:
1-\eta(q)<1-\eta(q)\cos (\xi \ln q) \leq \omega\right\}.
$$
\end{thm}

To formulate the next Theorem we introduce the restriction
operator
$$
\mathcal{R}_{\omega}^{S}: PW_{\omega}^{\natural}(G)\rightarrow
L_{2}^{\natural}(S), S\subset V(G),
$$
where
$$
\mathcal{R}_{\omega}^{S}(\varphi)=\varphi |_{S}, \varphi \in
PW_{\omega}^{\natural}(G).
$$
\begin{thm}
If $G$ is a homogeneous tree then for every $\omega>0$ and every
finite set $S\subset V(G)$ the restriction operator
$\mathcal{R}_{\omega}^{S}$ is surjective.

\end{thm}

\begin{proof}
Pick a finite set $S$ and assume that a function $\psi\in
L_{2}^{\natural}(S)$ is orthogonal to all restrictions to the set
$S$ of all functions from a space $PW_{\omega}^{\natural}(G),
1-\eta(q)<\omega<1+\eta(q)$.  It is known \cite{CMS} that the
functions $z\rightarrow \Phi_{z}(v)$ are entire functions for
every fixed $v\in V(G)$. Since $\psi$ has support on a finite set
$S$ its Helgason-Fourier transform $\mathcal{H}$ according to the
formula (5.1) is a finite combination of some functions
$\Phi_{\xi}(v), \xi\in [-\tau/2, \tau/2), v\in V(G)$ and in
particular an analytic function on the real line. It implies that
the set of zeros of $\mathcal{H}\psi$ has measure zero. At the
same time by the Parseval's relation this function should be
orthogonal to any function with support in the set $\Pi_{\omega}$
which is a  set of  positive measure. This shows that there is no
function in $L_{2}^{\natural}(S)$ which is orthogonal to
restrictions to $S$ of all functions  in
$PW_{\omega}^{\natural}(G)$. Because the set $L_{2}^{\natural}(S)$
is finite dimensional  it implies that it is exactly the space of
all restrictions of $PW_{\omega}^{\natural}(G)$ to $S$. The
Theorem is proved.
\end{proof}

On homogeneous trees of order $q+1, q\geq 2,$ the spectrum of
Laplacian $\mathcal{L}$ is separated from zero and in this case
the operator $\mathcal{L}^{s}, s\in \mathbb{R}$ can be used as the
operator $\mathcal{B}$ from the Corollary 3.1 with
$$
\|\mathcal{L}^{s}\|=\omega^{s},
\|\mathcal{L}^{-s}\|=\left(1-\eta(q)\right)^{-s}, q\geq 2, s\in
\mathbb{R}.
$$

The corresponding inequality with $\mathcal{L}^{s}, s\in
\mathbb{R},$ in place of $\mathcal{B}$ means that a function $f\in
PW_{\omega}(G)$ is uniquely determined by the values of its
"derivatives" $\mathcal{L}^{s}f, s\in \mathbb{R},$ on a uniqueness
set $U$.

\begin{thm}
If $G$ is a homogeneous tree then for any $s\in \mathbb{R}$ there
exists a frame  $\{\Psi_{u}^{(s)}\}_{u\in U}$ in the space
$PW_{\omega}(G)$ such that the following reconstruction formula
holds true for all $f\in PW_{\omega}(G)$
$$
f(v)=\sum_{u\in U}\mathcal{L}^{s}f(u)\Psi_{u}^{(s)}(v), v\in V(G).
$$
\end{thm}

In what follows we will obtain  explicit constants for the
Poincare inequality for some specific finite sets of point on
homogeneous trees. What is interesting that these sets have a
large volume and can be even infinite, but  corresponding Poincare
constants are less than one.

Consider a homogeneous tree of order $q+1$. We will say that the
root of this tree belongs to the level zero, the next $q$ vertices
belong to the level one, the next $q^{2}$ belong to the level two
and so on. A level of order $m$ will be denoted as $l_{m}$. Direct
computations show that the following Lemma holds true.

\begin{lem}
On  a homogeneous tree $G$  of order $q$ for any level $S=l_{m}$
of order $m$ the following Poincare inequality holds true
$$
\|\varphi\|\leq
\left(1+\frac{q}{(q+1)^{2}}\right)^{-1/2}\|\mathcal{L}\varphi\|,
\varphi\in L_{2}(l_{m}).
$$
In implies in particular that any finite or infinite set of the
following form
$$
S=\bigcup_{m=0} l_{3m}
$$
is a removable set for any space $PW_{\omega}(G)$ with any
$$
\omega<\alpha(q)=\left(1+\frac{q}{(q+1)^{2}}\right)^{1/2}>1.
$$

 Moreover, there exists a frame of functions
$\{\Theta_{u}\}_{u\in U}, U=V(G)\backslash S,$ in the space
$PW_{\omega}(G)$ such that the following reconstruction formula
holds true for all $f\in PW_{\omega}(G)$
$$
f(v)=\sum_{u\in U}f(u)\Theta_{u}(v), v\in V(G).
$$

Another way to reconstruct $f\in PW_{\omega}(G)$ is by using the
Theorem 1.4 with
$$
\Lambda=\left(1+\frac{q}{(q+1)^{2}}\right)^{-1/2}.
$$
\end{lem}
\bigskip

Note that according to the Lemma 3.5 for all functions  in
$PW_{\omega}(G)$ where $\omega<1$ any set of the form
$$
S=\bigcup_{m=0} l_{2m}
$$
is a removable set.

\bigskip

\section{Applications to  finite graphs}

In this section we consider finite graphs. For a set $S\subset
V(G)$ with a non-empty boundary $bS$ consider the set
$\mathcal{D}(S)$ of all functions $f$ from $L_{2}(G)$ which vanish
on the boundary $bS$. For a subset $S\subset V(G)$ the induced
subgraph determined by all edges that have both endpoints in $S$.
The first Dirichlet eigenvalue of an induced subgraph on $S$ is
defined as follows \cite{Ch}:
\begin{equation}
\lambda_{D}(S)=\inf_{f\in \mathcal{D}(S),f\neq
0}\frac{\left<f,\mathcal{L}f\right>}    {\left<f,f\right>}.
\end{equation}

 Because $L_{2}(S)\subset \mathcal{D}(S)$ we obtain the inequality

\begin{equation}
\|\varphi\|\leq
\frac{1}{\lambda_{D}(S)}\|\mathcal{L}_{G}\varphi\|, \varphi\in
L_{2}(S),
\end{equation}
which is different from (3.7).

Note that  since $L_{2}(S)$ is a subspace of $\mathcal{D}(S)$ the
constant in the inequality (6.2) is not the best possible. Only
when $\overline{S}=S\cup b S$ coincide with entire set $V(G)$ this
inequality  is exact.

If $\delta$ is the isoperimetric dimension of the graph $G$ then
there exists a constant $C_{\delta}$ which depends just on
$\delta$ (see \cite{Ch}) such that
\begin{equation}
\lambda_{D}(S)>C_{\delta}\left(\frac{1}{vol S}\right)^{2/\delta},
vol S=\sum_{v\in S}d(v).
\end{equation}

We obtain the following statement in which we use the same
notations as above.
\begin{thm} If a set $S\subset V(G)$ and an $\omega>0$
satisfy the inequality
$$
\omega<C_{\delta}\left(\frac{1}{vol S}\right)^{2/\delta}
$$
then the set $S$ is removable for  the space $PW_{\omega}(G)$.
\end{thm}
The following statement  gives a certain connection between
distribution of eigenvalues and existence of specific sets of
vertices.
\begin{thm}
If a finite graph $G$ has $N=|V(G)|$ vertices and a set $S\subset
V(G)$ is a set of type $\Lambda$ then there are at most $|U|$
eigenvalues (with multiplicities) of $\mathcal{L}$ on the interval
$[0,1/\Lambda)$ where $U=V\setminus S$ and there are at least
$N-|U|$ eigenvalues which belong to
 the interval $[1/\Lambda,\lambda_{\max}]$.
\end{thm}
\begin{proof}
If $S$ is a set of type $\Lambda$ then $U=V\setminus S$ is a
uniqueness set for any space $PW_{\omega}(G)$ with
$\omega<1/\Lambda$. It means that $|U|$ which is the dimension of
the space $L_{2}(U)$ cannot be less than the number of eigenvalues
(with multiplicities)  of $\mathcal{L}$ on the interval $[0,
1/\Lambda)$.
\end{proof}
The Theorem 6.2 implies the Corollaries 1.1 and 1.2 from the
Introduction.

\bigskip

 \textbf{Acknowledgement:} I would like to thank Dr. Meyer
 Pesenson for encouragement and many useful conversations.

\end{document}